\documentclass[11pt]{amsart}
\pdfoutput=1
\usepackage{geometry}   
\usepackage{graphicx}
\usepackage{amssymb}
\usepackage{epstopdf} 
\usepackage{graphicx, amsmath, amsthm, latexsym, amssymb, amsfonts, epsfig, url}

\newtheorem{theorem}{Theorem}[section]
\newtheorem{lemma}[theorem]{Lemma}
\newtheorem{proposition}[theorem]{Proposition}

\theoremstyle{definition}
\newtheorem{definition}[theorem]{Definition}

\newcommand\R{\mathbb R}
\newcommand\AssI{{\rm Ass}_n^{\rm I}}
\newcommand\AssII{{\rm Ass}_n^{\rm II}}
\newcommand\AssIII{{\rm Ass}_n^{\rm III}}

\title{Three non-equivalent realizations of the associahedron}

\author{Cesar Ceballos}
\address[Cesar Ceballos]{Inst.\ Mathematics, MA 6-2, TU Berlin, 10623 Berlin, Germany.}  \email{ceballos@math.tu-berlin.de}

\author{G\"unter M. Ziegler}
\address[G\"unter M. Ziegler]{Inst.\ Mathematics, MA 6-2, TU Berlin, 10623 Berlin, Germany.}  \email{ziegler@math.tu-berlin.de}

\date{}                                      

\thanks{The first author is supported by DFG via the Research Training Group ``Methods for Discrete Structures'';
the second author is partially supported by DFG. We are grateful to Carsten Lange and Anton Dochterman for 
helpful discussions and comments about this paper.}

\begin{document}
\maketitle

\begin{abstract}
We review three realizations of the associahedron that arise as secondary polytopes, 
from cluster algebras, and as Minkowski sums of simplices, and show that under any choice of parameters, the resulting associahedra are affinely non-equivalent.\\
{\bf Note:} the results of this preprint have been included in a more comprehensive paper, 
jointly with Francisco Santos, arXiv:1109.5544.
\end{abstract}

\section{Introduction}

The associahedron, also known as the Stasheff polytope, is a ``mythical" simple polytope that was first described as a combinatorial object by Stasheff in 1963, and was used to explore associativity of $H$-spaces. Three ``conceptual" constructions of the associahedron as a polytope, among numerous others, are: the associahedron as a secondary polytope due to Gelfand, Zelevinsky and Kapranov \cite{GZK90} \cite{GZK91} (see also \cite[Chap.~7]{GKZ94}), the associahedron associated to the cluster complex of type $A_n$ due to Chapoton, Fomin and Zelevinsky \cite{CFZ02}, and the associahedron as a Minkowski sum of simplices introduced by Postnikov in \cite{Po05}. Each one of these realizations depends on a large number of parameters that, a priori, might be chosen appropriately so that the three constructions produce equivalent objects. The main result of this paper is to show that regardless of 
how the parameters are chosen, the three realizations are affinely non-equivalent.

\begin{figure}[ht]
	\centering
	\includegraphics[width=.82\textwidth]{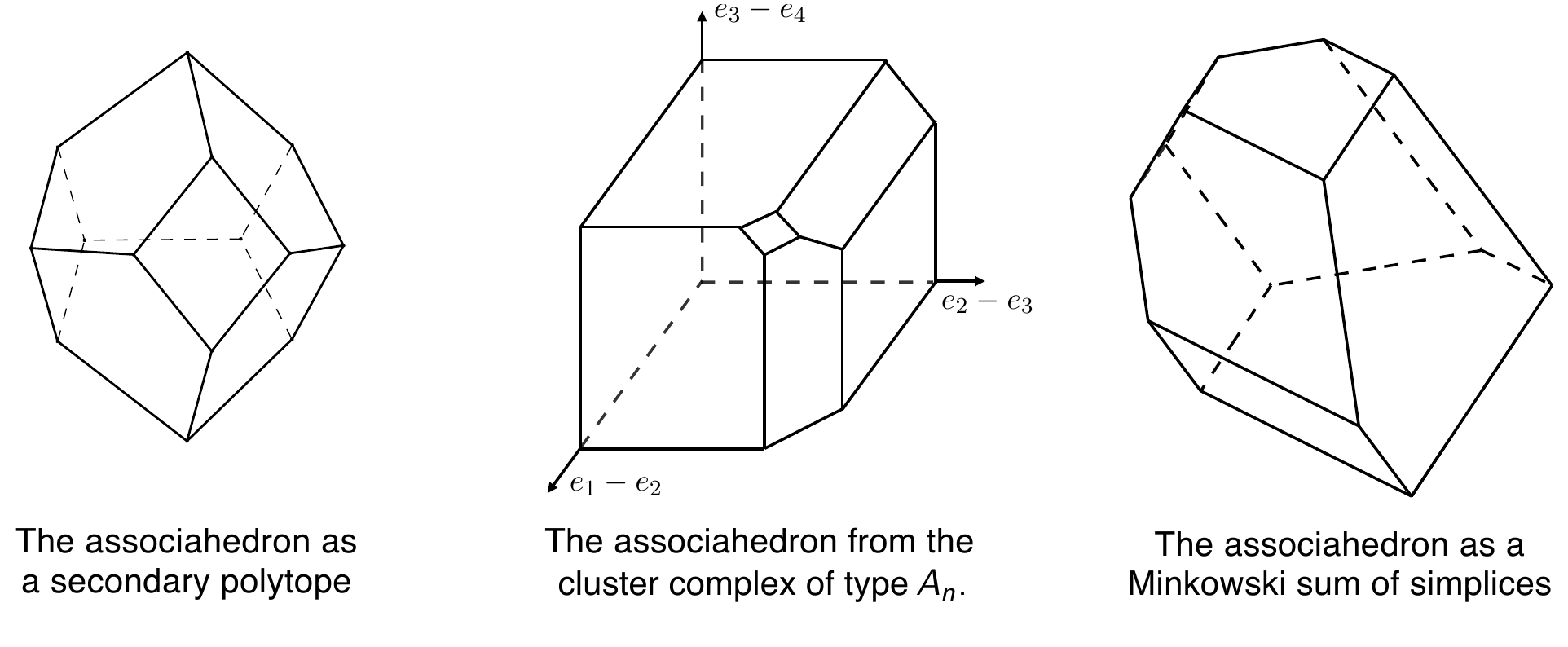}
\end{figure}

The paper is organized as follows. Section~\ref{review} starts with a quick review of associahedra. In Section~\ref{section_secondarypolytope} we review the construction of the associahedron as a secondary polytope of a convex polygon and prove in Theorem~\ref{theo_par_cont^I} that it has no pairs of parallel facets. Section \ref{section_clustercomplex} is a review of the construction of associahedra using the cluster complex of type $A_n$, which yields an $n$-dimensional associahedron with $n$ pairs of parallel facets. Theorem~\ref{theorem_Ass^II_parallel} describes these pairs in terms of pairs of diagonals of an $(n+3)$-gon. In Section~\ref{section_minkowski_sums} we review the construction of the $n$-dimensional associahedron as a Minkowski sum of simplices, and provide in Theorem~\ref{theorem_ass^III_correspondece} a precise description of a correspondence between faces of this polytope and subdivisions of an $(n+3)$-gon. We prove in Theorem~\ref{theorem_parallel_ass^III} that this associahedron has $n$ pairs of parallel facets, and we identify the corresponding pairs of diagonals of an $(n+3)$-gon. Finally, in Section~\ref{twotheorems} we show that all three types of realizations are affinely non-equivalent for any choice of parameters.
 
\tableofcontents

\section{Three realizations of the associahedron}\label{review}

We start by recalling the definition of an $n$-dimensional associahedron in terms of polyhedral subdivisions of an $(n+3)$-gon.

\begin{definition}
The \emph{associahedron} $\mathrm{Ass}_n$ is an $n$-dimensional simple polytope whose face lattice is isomorphic to the lattice of polyhedral subdivisions (without new vertices) of a convex $(n+3)$-gon ordered by refinement.
\end{definition}

\begin{figure}[ht] 
   \centering
   \includegraphics[width=0.6\textwidth]{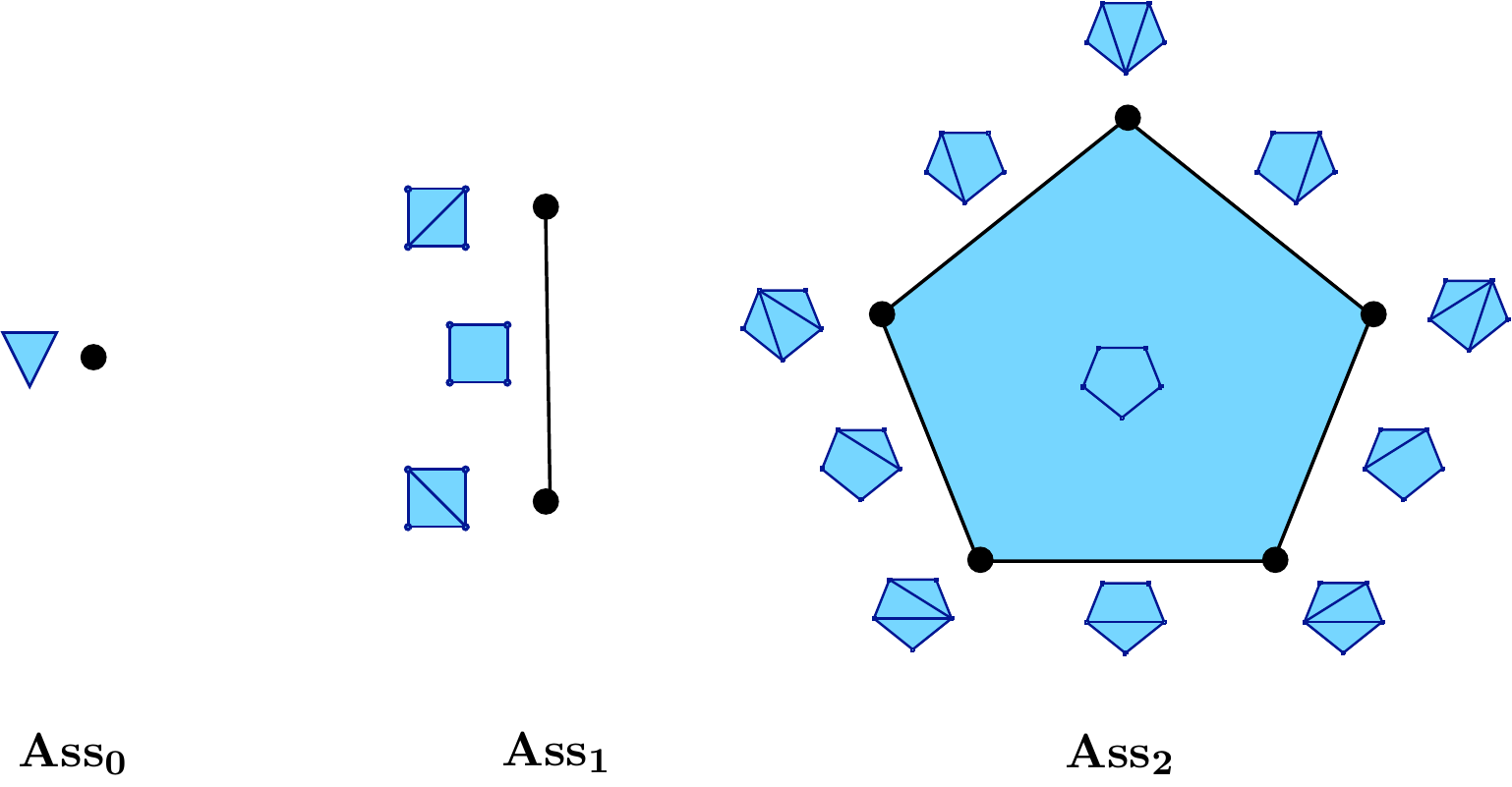} 
   \caption{The associahedron $\mathrm{Ass}_n$ for $n=1,2$ and 3.}
   \label{polyre1-1}
\end{figure}

Figure \ref{polyre1-1} illustrates the polytope $\mathrm{Ass}_n$ for $n=0,1$ and $2$. Such a polytope exists for every $n\ge0$. We will present three constructions for it,
which will be shown to be essentially different in the last section of this paper.

\subsection{Construction I: The associahedron as a secondary polytope}
\label{section_secondarypolytope}

The secondary polytope is an ingenious construction motivated by the theory of hypergeometric functions as developed by I.M. Gelfand, M. Kapranov and A. Zelevinsky \cite{GKZ94}. In this section we recall the basic definitions and main results related to this topic, which yield in particular that the secondary polytope of any convex $(n+3)$-gon is an $n$-dimensional associahedron. For more detailed presentations we refer to \cite[Lect 9]{Zi} 
and \cite[\S 5]{LoRaSa10}.

\subsubsection{The secondary polytope construction} 

\begin{definition}[GKZ vector/secondary polytope] \label{def_GKZ}
Let $Q$ be a $d$-dimensional convex polytope with $n$ vertices. The \emph{GKZ vector} 
$v(t)\in \R^n$ of a triangulation $t$ of $Q$ is 
\begin{eqnarray*}
v(t)\ := \ \sum_{i=1}^n \text{vol} (\text{star}_t(i)) e_i 
\ \ =\ \ \sum_{i=1}^n \sum_{\sigma\in t\,:\, i\in\sigma}    \text{vol} (\sigma) e_i
\end{eqnarray*} 
The \emph{secondary polytope} of $Q$ is defined as
\[
\Sigma (Q)\ :=\ \text{conv}\{ v(t) : t \text{ is a triangulation of } Q \}.
\]
\end{definition}

\begin{proposition} [Gelfand--Kapranov--Zelevinsky  \cite{GZK90}] \label{theoGKZ}
Let $Q$ be a $d$-dimensional convex polytope with $n$ vertices. Then
the secondary polytope $\Sigma (Q)$ has the following properties:
\begin{enumerate}
\item $\Sigma (Q)$ is an $(n-d-1)$-dimensional polytope.
\item The vertices of $\Sigma (Q)$ are in bijective correspondence with the regular triangulations of $Q$ without new vertices.
\item The faces of $\Sigma (Q)$ are in bijective correspondence with the regular subdivisions of~$Q$.
\item The face lattice of $\Sigma (Q)$ is isomorphic to the lattice of regular subdivisions of $Q$, ordered by refinement.
\end{enumerate}
\end{proposition} 

\subsubsection{The associahedron as the secondary polytope of a convex $(n+3)$-gon}

\begin{definition} $\AssI(Q)\subset \R^{n+3}$ is defined as
the ($n$-dimensional) secondary polytope of a convex $(n+3)$-gon $Q\subset\R^2$:
\[
\AssI(Q):=\Sigma (Q).
\]
\end{definition}

\begin{proposition} [GKZ]
$\AssI(Q)$ is an $n$-dimensional associahedron.
\end{proposition}

\begin{theorem}\label{theo_par_cont^I}
$\AssI(Q)$ has no parallel facets for $n\geq 2$.
\end{theorem}

\begin{proof}
The polytope $\AssI(Q)$ satisfies the following correspondence: 
\[
\begin{array}{rcl}
 \text{vertices}  & \leftrightarrow & \text{triangulations of } Q \\
 \text{facets}  & \leftrightarrow & \text{diagonals of } Q  \\
\end{array}
\]
For a given diagonal $\delta$ of $Q$, we denote by $F_\delta$ the facet of $\AssI(Q)$ corresponding  to $\delta$, and we define the $(n-1)$-dimensional subspace $V_\delta \subset \R^{n+3}$ by
\[
V_\delta= \mathrm{span} \{ u-v : u,v \text{ are vertices of } F_\delta \}.
\]
Thus for diagonals $\delta$ and $\delta'$ of $Q$ we have  
\[
F_\delta \text{ is parallel to } F_{\delta'} \text{ if and only if } V_\delta=V_{\delta'}.
\]  
Given two different diagonals $\delta$ and $\delta'$ of $Q$ there are two cases:
\smallskip

\noindent
\emph{Case 1.}  $\delta$ and $\delta'$ do not cross, i.e they do not intersect in an interior point of $Q$. In this case, $F_\delta$ and $F_{\delta'}$ intersect on the face of $\AssI(Q)$ corresponding to the subdivision generated by the two non crossing diagonals $\{\delta,\delta'\}$, and so $F_\delta$ and $F_{\delta'}$ are not parallel.

\begin{figure}[ht] 
   \centering
   \includegraphics[width=0.85\textwidth]{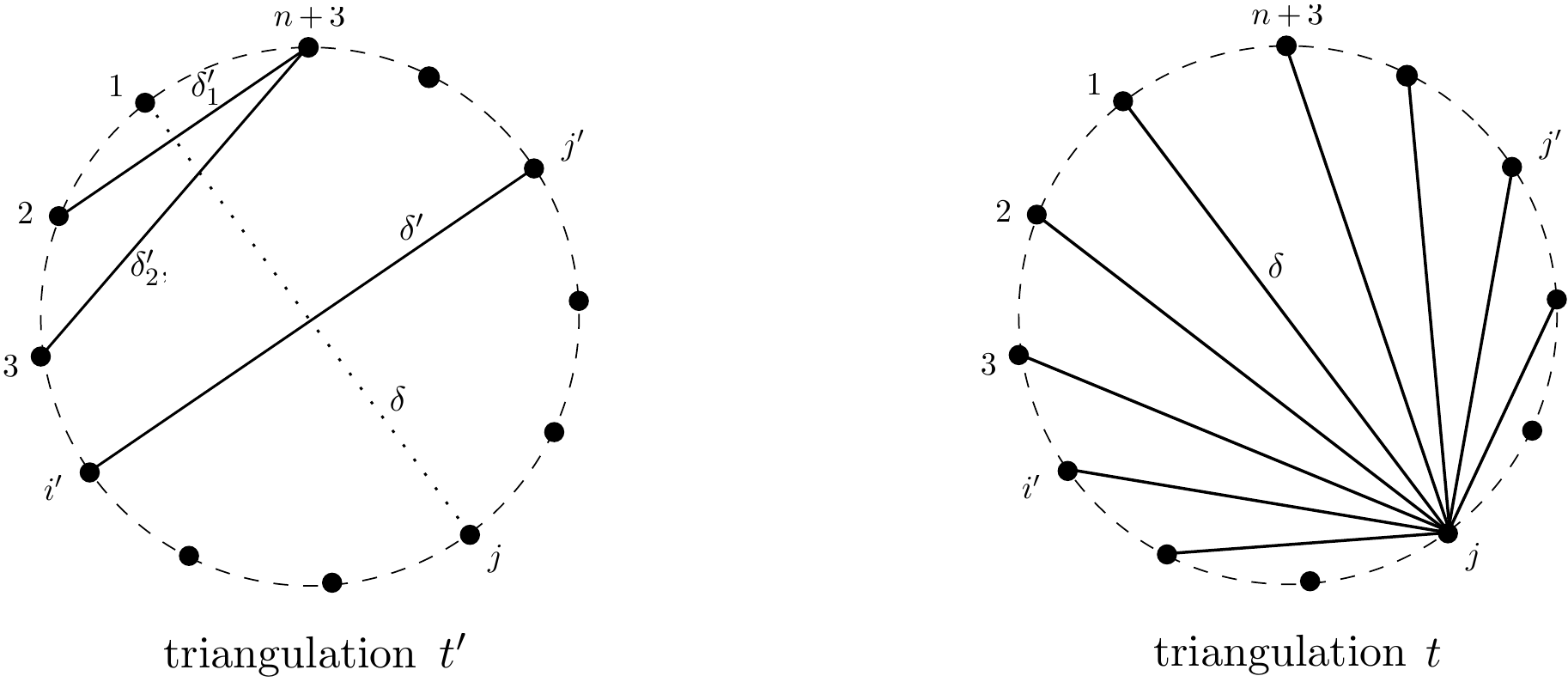} 
   \caption{Diagonals mentioned in Case 2.}
   \label{sec_parallel_proof}
\end{figure} 

\noindent
\emph{Case 2.} $\delta$ and $\delta'$ cross in an interior point of $Q$. In this case, using $n+3\geq 5$, we can relabel the vertices of $Q$ in either clockwise or counterclockwise direction so that $1\in \delta$, $2\notin \delta'$, so $\delta=\{1,j\}$ and $\delta'=\{i',j' \}$ with $3 \leq i' < j< j' \leq n+3$. Thus there is 
a triangulation $t'=(\delta_1',\delta_2',\dots , \delta_n')$ of $Q$ with $\delta_1'=\{2,n+3\}$, $\delta_2'=\{3,n+3\}$ and $\delta' \in \{\delta_2',\delta_3',\dots ,\delta_n'\}$ (see left hand side of Figure \ref{sec_parallel_proof}). 

Let $t_1'$ be the triangulation obtained by flipping the diagonal $\delta_1'$ of $t'$ (to $\{1,3\}$), and let 
\[
w=v(t_1')-v(t').
\]
Since the triangulations $t_1'$ and $t'$ contain the diagonal $\delta'$, the corresponding vertices $v(t_1')$ and $v(t')$ are both vertices of $F_{\delta'}$, which implies that 
\[
w\in V_{\delta'}.
\] 

On the other hand, we prove that $w\notin V_\delta$. For this, let $t$ be the triangulation of $Q$ determined by the set of non-crossing diagonals of the form $\{k,j\}$ with $k\in S$, where $S=\{k\in [n+3]: k\neq j-1,j,j+1\}$ (as shown in the right hand side of Figure \ref{sec_parallel_proof}), and let $t_{k,j}$ be the triangulation obtained by flipping the diagonal $\{k,j\}$ of $t$. We denote by $v(t)$ and $v(t_{k,j})$ the vertices of $\AssI(Q)$ corresponding to $t$ and $t_{k,j}$, respectively. Then we have 
\[
V_\delta= \mathrm{span} \{ u_k= v(t_{k,j})-v(t) : k\in S \text{ and } k\neq 1 \},
\]

Suppose that $w\in V_\delta$. Then $w$ can be written as a linear combination 
\begin{equation}\label{eq1}
w=\sum_{k\in S\backslash \{1\} } c_k u_k
\end{equation}
Definition \ref{def_GKZ} yields that each of the vectors $w$ and $u_k$ has exactly $4$
non-zero coordinates. More precisely, if $w^\ell$ is the $\ell$-coordinate of $w$ 
and $u_k^\ell$ is the $\ell$-coordinate of $u_k$, then 
\begin{itemize}
\item $w^\ell \neq 0$ for $\ell \in \{ 1,2,3,n+3 \}$ and $w^\ell=0$ otherwise.
\item $u_{n+3}^\ell \neq 0$ for $\ell \in \{ n+2,n+3,1,j \}$ and $u_{n+3}^\ell=0$ otherwise.
\item for $k\in S\backslash \{1,n+3\}$, $u_k^\ell \neq 0$ for $\ell \in \{ k-1,k,k+1,j \}$ and $u_k^\ell =0$ otherwise.
\end{itemize}  
We will respresent the system of linear equations (\ref{eq1}) in matrix form.
The columns of the matrix are the vectors 
$\{u_k\}_{k\in S\backslash \{1\}}$, 
$c=(c_k)_{k\in S\backslash \{1\}}$ is the vector of coefficients. 
The symbol~$*$ represents non-zero entries. 
\begin{center}
	\includegraphics[width=1\textwidth]{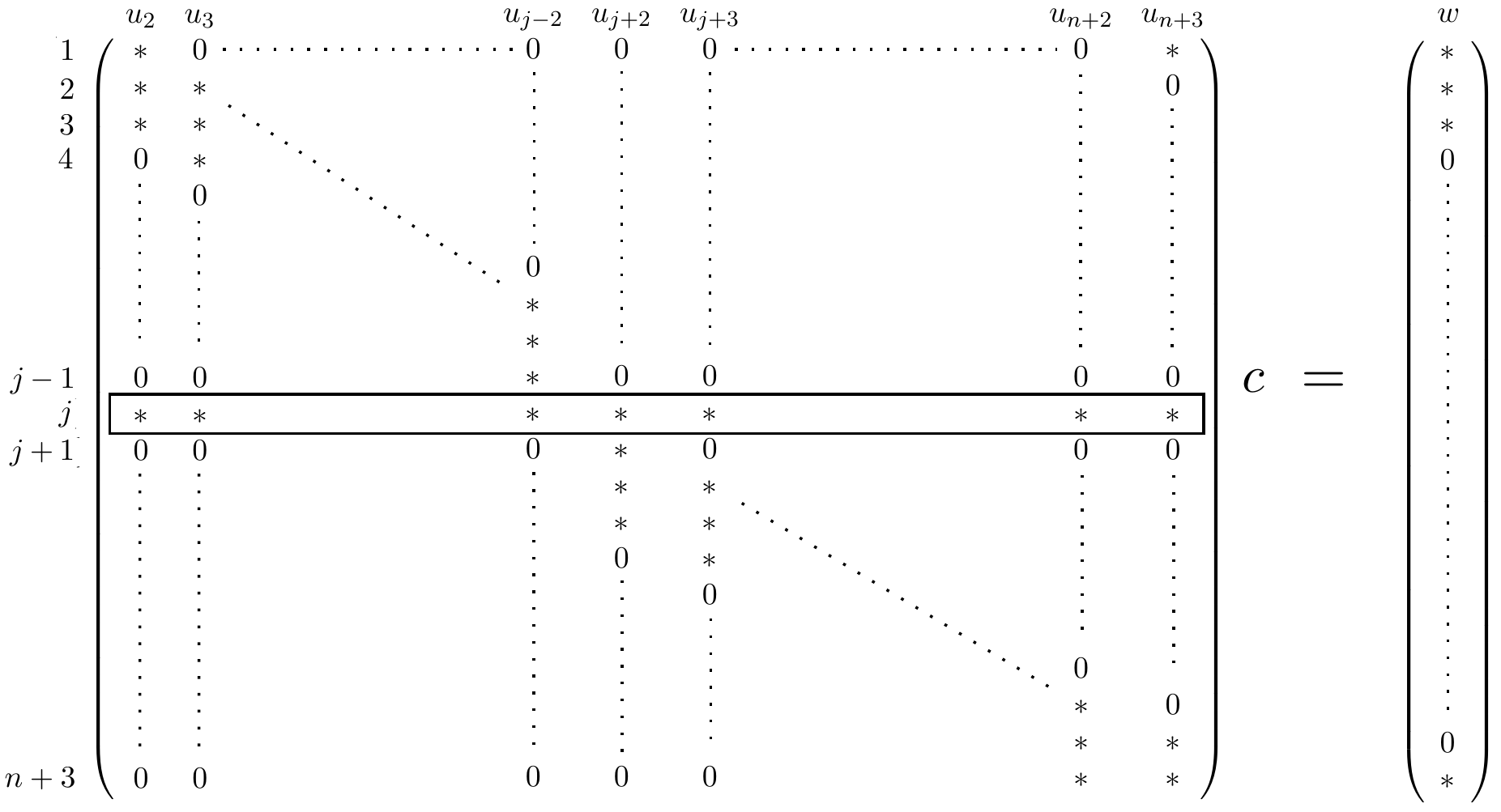}
\end{center}
One can easily see that the coefficients $c_k$ for $j+2 \leq k \leq n+3$ must be equal to~$0$,
and similarly for $j-2\ge k\ge 3$. Thus $w=c_2u_2$, which yields a contradiction.  
  
Thus we know that $w\in V_{\delta'}$ but $w\notin V_\delta$, so $V_\delta \neq V_{\delta'}$,
 and hence $F_\delta$ and $F_{\delta'}$ are not parallel.
\end{proof}

\subsection{Construction II: The associahedron associated to a cluster complex}
\label{section_clustercomplex}

Cluster complexes are combinatorial objects that arose in the theory of cluster algebras
\cite{FZ02} \cite{FoZe03}. They correspond to the normal fans 
polytopes known as generalized associahedra. In this section we introduce the particular case of type $A_n$, which is related to the classical associahedron. Most of this theory was introduced by S. Fomin and A. Zelevinsky and can be found in \cite{FZ03}, \cite{FR07} and \cite{CFZ02}.   

\subsubsection{The cluster complex of type $A_n$}
The \emph{root system of type} $A_n$ is the set $\Phi := \Phi (A_n) = \{ e_i-e_j, \ 1\leq i \neq j \leq n+1 \}$. The \emph{simple roots} of type $A_n$ are the elements of the set $\Pi =\{ \alpha_i=e_i-e_{i+1}, i\in [n] \}$, the set of \emph{positive roots} is 
$\Phi_{>0}= \{e_i-e_j:i<j\}$, and the set of \emph{almost positive roots} is $\Phi_{\geq -1}:=\Phi_{>0}\cup -\Pi$.

There is a natural correspondence between the set $\Phi_{\geq-1}$ and diagonals of the $(n+3)$-gon $P_{n+3}$: We identify the negative simple roots $-\alpha_i$
\begin{figure}[ht]
	\centering
	\includegraphics[width=0.3\textwidth]{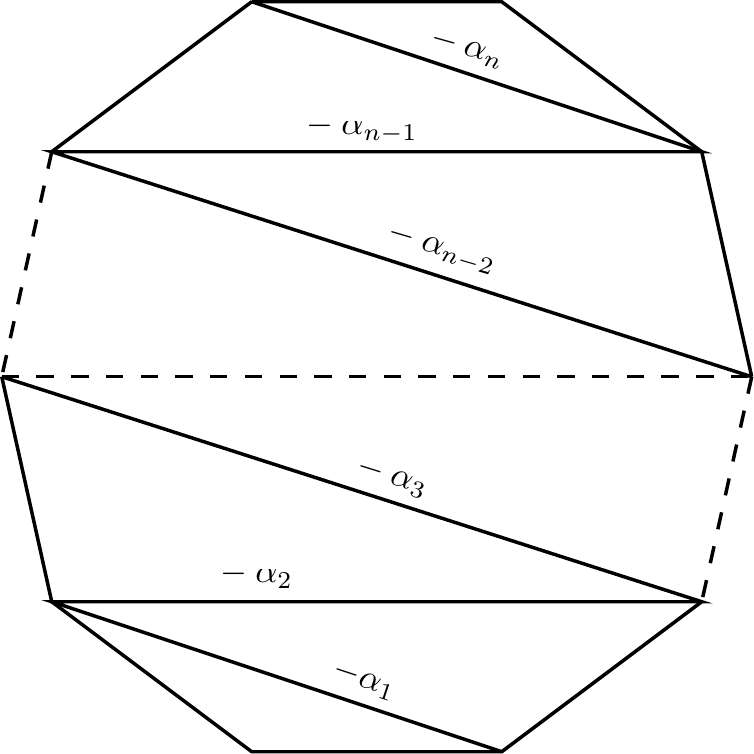}
	\caption{Snake and negative roots of type $A_n$.}
	\label{snake}
\end{figure}
with the diagonals on the snake of $P_{n+3}$ as illustrated in Figure \ref{snake}. Each positive root is a consecutive combination 
\[
\alpha_{ij}=\alpha_i+\alpha_{i+1}+\cdots + \alpha_j, \hspace{1cm} 1\leq i \leq j \leq n,
\]
and thus is identified with the unique diagonal of $P_{n+3}$ crossing the (consecutive) diagonals that correspond to $-\alpha_i, -\alpha_{i+1}, \cdots, -\alpha_{j}$.

\begin{definition}[Cluster complex of type $A_n$]
Two roots $\alpha$ and $\beta$ in $\Phi_{\geq -1}$ are \textit{compatible} if 
their corresponding diagonals do not cross. The \emph{cluster complex} $\Delta (\Phi)$ of type $A_n$ is the clique complex of the compatibility relation on $\Phi_{\geq -1}$, i.e., the complex whose simplices correspond to the sets of almost positive roots that are pairwise compatible. Maximal simplices of $\Delta (\Phi)$ are called \emph{clusters}.
\end{definition}

In this case, the cluster complex satisfies the following correspondence,
which is dual to the complex of the classical associahedron $\mathrm{Ass}_n$:
\[
\begin{array}{rcl}
 \text{vertices} & \leftrightarrow & \text{diagonals of a convex } (n+3) \text{-gon} \\
 \text{simplices} & \leftrightarrow & \text{polyhedral subdivisions of the } (n+3) \text{-gon}     \\
 	& & \text{(viewed as collections of non-crossing diagonals)} \\
 \text{maximal simplices} & \leftrightarrow & \text{triangulations of the } (n+3) \text{-gon}    \\
	& & \text{(viewed as collections of } n \text{ non-crossing diagonals)}
\end{array}
\]

\noindent
The following proposition is the particular case of type $A_n$ of \cite[Thm.~1.10]{FZ03}. 
It allows us to think of the cluster complex as the complex of a complete simplicial fan.  

\begin{proposition}\label{th_simpfan}
The simplicial cones $\R_{\geq 0} C$ generated by all clusters $C$ of type $A_n$ 
form a complete simplicial fan in the ambient space 
$\{(x_1,\dots ,x_{n+1})\in \R^{n+1}: x_1+\cdots +x_{n+1}=0 \}$.
\end{proposition}

\begin{proposition}\label{theorem_cluster}
The simplicial fan in {\normalfont Proposition~\ref{th_simpfan}} is the normal fan of a simple $n$-dimensional polytope $P$.
\end{proposition}

Proposition \ref{theorem_cluster} was first conjectured by Fomin and Zelevinsky
\cite[Conj.~1.12]{FZ03} and proved by Chapoton, Fomin, and Zelevinsky \cite{CFZ02}. 
For an explicit description of such a polytope by inequalities see \cite[Cor.~1.9]{CFZ02}. 

\subsubsection{The associahedron $\AssII(A_n)$}

\begin{definition}
$\AssII(A_n)$ is any polytope whose normal fan is the fan with maximal cones $\R_{\geq 0} C$ generated by all clusters $C$ of type $A_n$.
\end{definition}

\begin{proposition}
$\AssII(A_n)$ is an $n$-dimensional associahedron.
\end{proposition}

A polytopal realization of the associahedron $\mathrm{Ass}_2^{II}(A_2)$ is illustrated in
Figure \ref{fig_assA2}; note how the facet normals correspond to the almost positive roots 
of~$A_2$.

\begin{figure}[ht]
	\centering
	\includegraphics[width=0.4\textwidth]{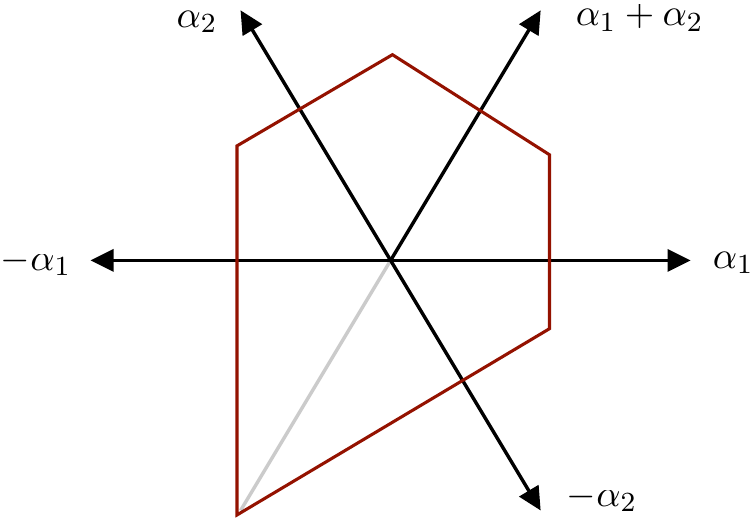}
	\caption{The simplicial fan of the cluster complex of type $A_2$ and the associahedron $\mathrm{Ass}(A_2)$}
	\label{fig_assA2}
\end{figure}

\begin{theorem}\label{theorem_Ass^II_parallel}
$\AssII(A_n)$ has exactly $n$ pairs of parallel facets. These correspond to the pairs of roots $\{ \alpha_i,-\alpha_i \}$, for $i=1,\cdots , n$, or equivalently, to the pairs of diagonals $\{ \alpha_i,-\alpha_i \}$ as indicated in {\normalfont Figure~\ref{Parallel_ass^II}}.
\end{theorem}

\begin{figure}[ht] 
	\centering
	\includegraphics[width=0.85\textwidth]{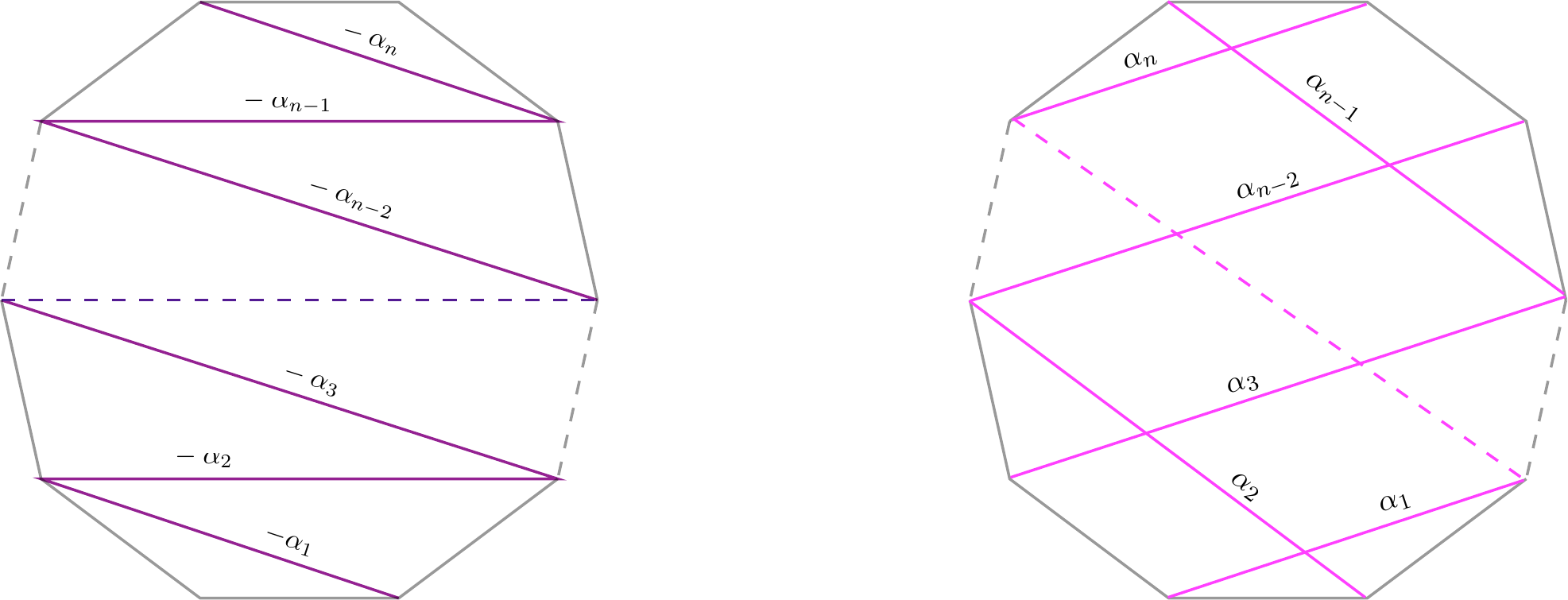}
	\caption{The diagonals of the $(n+3)$-gon that correspond to the pairs of parallel facets of $\AssII(A_n)$.}
	\label{Parallel_ass^II}
\end{figure}

\subsection{Construction III: The associahedron as a Minkowski sum of simplices}
\label{section_minkowski_sums}

The realization of the associahedron as a Minkowski sum of simplices  generalizes Loday's realization introduced in \cite{Lo04}, and is a special case of generalized permutahedra, 
an important family of polytopes studied by Postnikov \cite{Po05}. 
Loday's realization is also equivalent to the associahedra described recently by Buchstaber~\cite{Bu08}. 

\begin{definition}
For any vector $\mathbf{a}=\{\mathrm{a}_{ij}>0: 1\leq i\leq j \leq n+1\}$ of positive parameters let
\[
\AssIII(\mathbf{a}):=\sum_{1\leq i\leq j\leq n+1}\mathrm{a}_{ij}\Delta_{[i,\dots ,j]},
\]
where $\Delta_{[i,\dots ,j]}$ denotes the simplex conv$\{e_i,e_{i+1},\dots , e_j\}$ in $\R^{n+1}$. 
\end{definition}

\begin{proposition}[Postnikov {\cite[\S 8.2]{Po05}}] 
$\AssIII(\mathbf{a})$ is an $n$-dimensional associahedron. In particular, 
for $a_{ij}\equiv1$ this yields the 
realization of Loday~\cite{Lo04}.
\end{proposition}

In this section we present two main results. Theorem \ref{theorem_ass^III_correspondece} explains a correspondence between faces of $\AssIII(\mathbf{a})$ and subdivisions of a convex $(n+3)$-gon, and Theorem \ref{theorem_parallel_ass^III} is a combinatorial description of the pairs of parallel facets of $\AssIII(\mathbf{a})$.

\subsubsection{Correspondence between faces of the polytope and subdivisions of an $(n+3)$-gon}

Each face of $\AssIII(\mathbf{a})$ can be represented as the set of points in $\AssIII(\mathbf{a})$ maximizing a linear function. Figure \ref{fig_corresp} illustrates an example of a natural correspondence between faces of $\AssIII(\mathbf{a})$ maximizing a linear function, and subdivisions of an $(n+3)$-gon. This correspondence works as follows. Let $w=(w_1,\dots , w_{n+1})\in (\R^{n+1})^*$ be a linear function and $F_w$ be the face of $\AssIII(\mathbf{a})$ that maximizes it. The subdivision $S_w$ that correspond to $F_w$ is a subdivision of an $(n+3)$-gon with vertices labeled in counterclockwise direction from $0$ to $n+2$; 
it is obtained as the common refinement of the subdivisions that are induced by
the polygons $H_k$, $k=0,\dots,n+2$, 
where $H_k=\text{conv}\{ i\in \{0,\dots,n+2\} : w_i\geq w_k \}$ with $w_0=w_{n+2}=\infty$ 
(see Figure~\ref{fig_corresp}).

\begin{figure}[ht]
	\centering
	\includegraphics[width=0.8\textwidth]{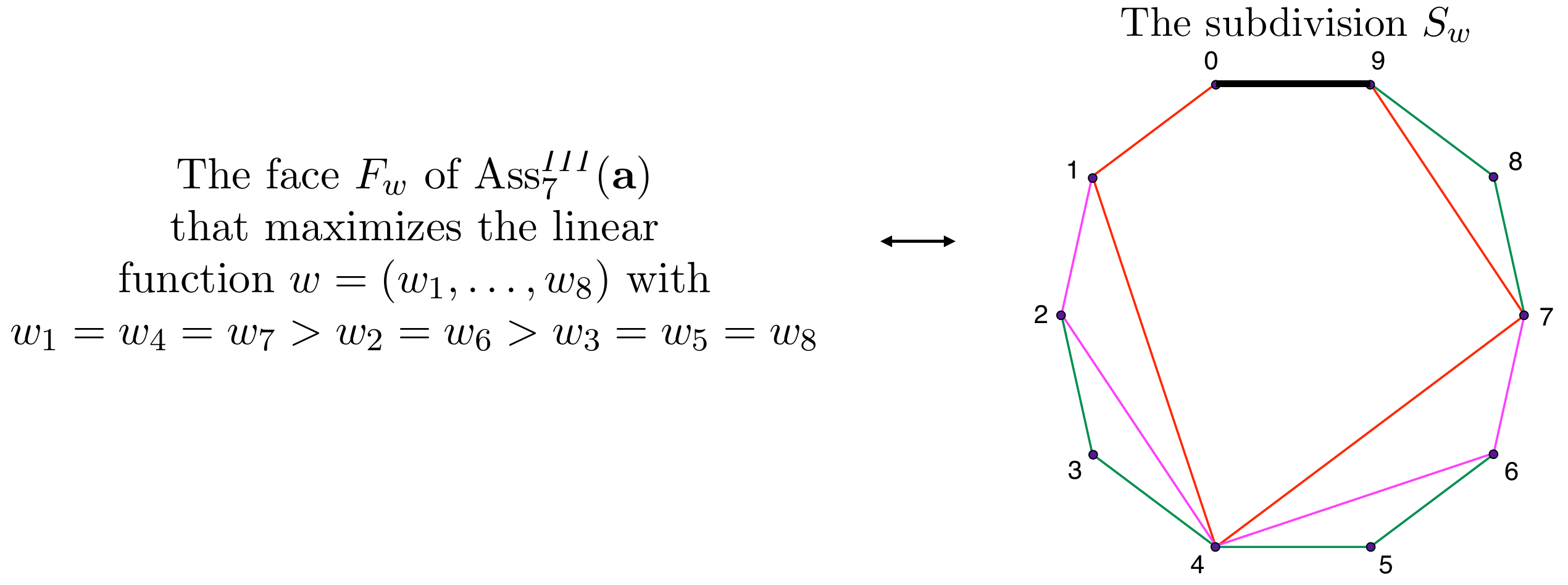}
	\caption{Example of the correspondence between faces of $\mathrm{Ass}_7^{\rm III}(\mathbf{a})$ maximizing a linear function and subdivisions of a convex $(7+3)$-gon.}
	\label{fig_corresp}
\end{figure}

\begin{theorem}\label{theorem_ass^III_correspondece} 
Let $\mathrm{LAss}_n^{\rm III}(\mathbf{a})$ be the face lattice of $\AssIII(\mathbf{a})$ and $\mathrm{LAss}_n$ be the lattice of subdivisions of a convex $(n+3)$-gon ordered by refinement. Then, the correspondence
\[
\begin{array}{rccc}
 \varphi : & \mathrm{LAss}_n^{\rm III}(\mathbf{a})  & \longrightarrow  & \mathrm{LAss}_n\\
  & F_w  & \longrightarrow  & S_w
\end{array}
\]
defines an order-preserving bijection. In particular, 
$\AssIII(\mathbf{a})$ is an associahedron. 
\end{theorem}

\subsubsection{Combinatorial description of the pairs of parallel facets}

\begin{theorem} \label{theorem_parallel_ass^III}
$\AssIII(\mathbf{a})$ has $n$ pairs of parallel facets. They correspond to the pairs of diagonals $(\{ n+2,i \} , \{ 0,i+1 \})$ for $1\leq i \leq n$, as illustrated in 
{\normalfont Figure~\ref{Parallel_ass^III}}.
\end{theorem}

\begin{figure}[ht]
	\centering
	\includegraphics[width=0.7\textwidth]{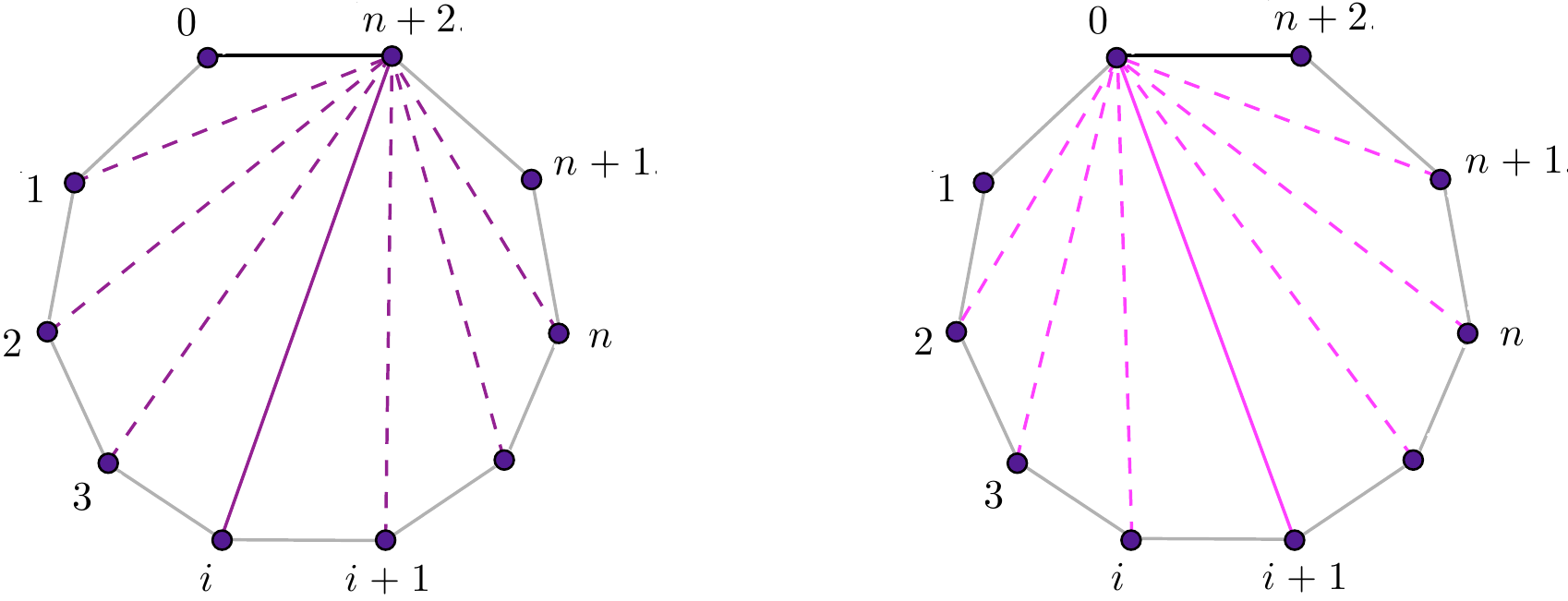}
	\caption{Diagonals of the $(n+3)$-gon corresponding to the pairs of parallel facets of $\AssIII(\mathbf{a})$}
	\label{Parallel_ass^III}
\end{figure}

\begin{proof}
A convex $(n+3)$-gon has three types of diagonals:
\begin{enumerate}
\item The diagonals of the form $\{n+2,i\}$, $1\leq i \leq n$. They correspond to facets of $\AssIII(\mathbf{a})$ maximizing the linear functions $w=(w_1,\dots , w_{n+1})$ with $w_1=\cdots = w_i > w_{i+1}=\cdots=w_{n+1}$. These facets are parallel to the affine spaces generated by $\Delta_{[1,\cdots , i]}+\Delta_{[i+1,\cdots , n+1]}$.
\item The diagonals of the form $\{0,i+1\}$, $1\leq i \leq n$. They correspond to facets of $\AssIII(\mathbf{a})$ maximizing the linear functions $w=(w_1,\dots , w_{n+1})$ with $w_{i+1}=\cdots=w_{n+1}>w_1=\cdots = w_i $. These facets are parallel to the affine spaces generated by $\Delta_{[1,\cdots , i]}+\Delta_{[i+1,\cdots , n+1]}$.
\item The diagonals of the form $\{i,j\}$, $1\leq i < i+1< j \leq n+1$.  They correspond to the facets of $\AssIII(\mathbf{a})$ maximizing the linear functions $w=(w_1,\dots , w_{n+1})$ with $w_1=\cdots = w_i=w_j=\cdots = w_{n+1}> w_{i+1}=\cdots = w_{j-1}$. These facets are parallel to the affine spaces generated by $\Delta_{1,\cdots i,j,\cdots ,n+1}+\Delta_{i+1,\cdots,j-1}$.
\end{enumerate}

\noindent
Using this information one sees that $\AssIII(\mathbf{a})$ has exactly $n$ pairs of parallel facets, and that they correspond to the pairs of diagonals $(\{ n+2,i \} , \{ 0,i+1 \})$.
\end{proof}

\section{The three types of associahedra are affinely non-equivalent} 
\label{twotheorems}

We have described three realizations of the associahedron together with combinatorial descriptions of the pairs of parallel facets. Each of the three constructions depends on a large number of parameters: 
\begin{itemize}
\item $\AssI(Q)$ depends on $2(n+3)$ parameters corresponding to the coordinates of the convex $(n+3)$-gon $Q\subset \R^2$.
\item $\AssII( A_n )$ depends on $\binom{n+1}2+n=\frac{n(n+3)}{2}$ parameters corresponding to the almost positive roots of type $A_n$, which determine the distances of the corresponding
hyperplanes from the origin.
\item $\AssIII( \mathbf a )$ depends on $\binom{n+2}2$ 
parameters corresponding to the values of $a_{ij}$, $1\leq i \leq j \leq n+1$.
\end{itemize}
One may wonder (and the second author has repeatedly asked) how can one choose 
the parameters in an appropriate way so that these constructions would 
produce the same result. 
The main results of this paper, given by Theorem \ref{theorem 1} and Theorem \ref{theorem 2},
show that this cannot be achieved at all: The three constructions 
produce affinely non-equivalent polytopes, for all choices of parameters. 

\begin{theorem}\label{theorem 1}
For $n\geq 2$, $\AssI(Q)$ is affinely equivalent to neither $\AssII( A_n )$ nor  
$\AssIII( \mathbf a )$.
\end{theorem}
 
\begin{proof}
By Theorem \ref{theo_par_cont^I}, $\AssI(Q)$ has no pairs of parallel facets, while 
$\AssII( A_n )$ and  $\AssIII( \mathbf a )$ do. Since affine maps preserve pairs of 
parallel facets, $\AssI(Q)$ is affinely equivalent to neither $\AssII( A_n )$ nor 
$\AssIII( \mathbf a )$.
\end{proof}

\begin{theorem} \label{theorem 2} 
$\AssII( A_n )$ and  $\AssIII( \mathbf a )$ are not affinely equivalent for $n\ge3$.
\end{theorem}

\begin{proof}
We say that a facet $F$ of a polytope $P$ is \emph{special} if there is 
an other facet $F'$ of $P$ which is parallel to $F$. Theorem \ref{theorem 2} follows   
from the following lemma.
\end{proof}

\begin{lemma} ~
\begin{enumerate}
\item $\AssII( A_n )$ and  $\AssIII( \mathbf a )$ both have $2n$ special facets.
\item There are two special facets of $\AssIII( \mathbf a )$ which intersect 
exactly $n-1$ other special facets. 
\item For $n>3$, every special facet of $\AssII( A_n )$ intersects 
more than $n-1$ special facets.\\
For $n=3$, there is only one facet which intersects exactly $n-1=2$ special facets.
\end{enumerate}
\end{lemma}

\begin{proof}~\\ 
(1) This has already been proved in Theorem \ref{theorem_Ass^II_parallel} and Theorem \ref{theorem_parallel_ass^III}. \\
(2) Two special facets intersect if and only if the corresponding \emph{special diagonals} do not cross. Figure \ref{Parallel_ass^III} shows the special diagonals associated to the construction of $\AssIII( \mathbf a )$. Among these, $\{n+2,1\}$ and $\{0,n+1\}$ have the property that they do not cross exactly $n-1$ other special diagonals. Therefore, the two special facets of $\AssIII( \mathbf a )$ corresponding to these two diagonals satisfy the desired condition.  \\
(3) The special facets of $\AssII( A_n )$ correspond to the diagonals of the form $\alpha_i$ or $-\alpha_i$, $i=1,\dots,n$, that are shown in Figure \ref{Parallel_ass^II}. 
A diagonal $-\alpha_i$ does not cross the $2(n-1)$ diagonals of the form $\alpha_j$ or $-\alpha_j$ with $j\neq i$. On the other hand, when $n>3$, a diagonal $\alpha_i$ does not cross the $n-1$ diagonals $-\alpha_j$ with $j\neq i$ and neither one between $\alpha_1$ and $\alpha_n$. Therefore, when $n>3$ any special facet of $\AssII( A_n )$ intersects more than $n-1$ other special facets.
In the case $n=3$, the facet corresponding to the diagonal $\alpha_2$ is the only special facet that intersects exactly $n-1=2$ other special facets.   
\end{proof}
 


\begin{thebibliography}{99}

\bibitem{BS92} Louis J. Billera and Bernd Sturmfels. \textit{Fiber polytopes.} Annals of Math. 135 (1992), 527-549.

\bibitem{BS94} Louis J. Billera and Bernd Sturmfels. \textit{Iterated fiber polytopes.} Mathematika 41 (1994), 348-363.

\bibitem{Bu08} Victor M. Buchstaber. \textit{Lectures on Toric Topology}. Toric Topology Workshop, KAIST 2008, Trends in Mathematics, Information Center of Mathematical Sciences, Vol.~11, No.~1, 2008, 1-55.   

\bibitem{CFZ02} Fr\'ed\'eric Chapoton, Sergey Fomin and Andrei Zelevinsky. \textit{Polytopal realizations of generalized associahedra.} Can. Math. Bull. {45} (2002), 537-566.

\bibitem{LoRaSa10} Jes\'us A. De Loera, J\"org Rambau and Francisco Santos. \textit{Triangulations: Structures for Algorithms and Applications.} Algorithms and Computation in Mathematics, Springer-Verlag, to appear. 

\bibitem{FR07} Sergey Fomin and Nathan Reading. \textit{Root systems and generalized associahedra.} in: ``Geometric Combinatorics'' (E. Miller, V. Reiner, B. Sturmfels, eds.)
IAS/Park City Math. Ser., Vol.~13, Amer. Math. Soc., Province, RI, 2007, 63-131.

\bibitem{FZ02} Sergey Fomin and Andrei Zelevinsky. \textit{Cluster algebras I: Foundations.} J. Amer. Math. Soc.  {15}  (2002), 497-529.

\bibitem{FoZe03} Sergey Fomin and Andrei Zelevinsky. \textit{Cluster algebras II: Finite type classification.} Invent. Math.  {154}  (2003), 63-121.

\bibitem{FZ03} Sergey Fomin and Andrei Zelevinsky. \textit{$Y$-systems and generalized associahedra}. Annals of Math. 158 (2003), 977-1018.

\bibitem{GKZ94} Israel M. Gelfand, Mikhail M. Kapranov and Andrei V. Zelevinsky. \textit{Discriminants, Resultants, and Multidimensional Determinants.} Birkh\"auser, Boston 1994.

\bibitem{GZK90} Israel M. Gelfand, Andrei V. Zelevinsky and Mikhail M. Kapranov. \textit{Newton polytopes of principal $A$-determinants.} Soviet Math. Doklady 40 (1990), 278-281.

\bibitem{GZK91} Israel M. Gelfand, Andrei V. Zelevinsky and Mikhail M. Kapranov. \textit{Discriminants of polynomials in several variables and triangulations of Newton polyhedra.} Leningrad Math. J. 2 (1991), 449-505.

\bibitem{Ha84} Mark Haiman. \textit{Constructing the associahedron}. 
Unpublished manuscript, MIT 1984, 11 pages,
\url{http://math.berkeley.edu/~mhaiman/ftp/assoc/manuscript.pdf}.

\bibitem{Lee89} Carl W. Lee. \textit{The associahedron and triangulations of the $n$-gon.} European J. Combinatorics 10 (1989), no. 6, 551-560.

\bibitem{Lo04} Jean L. Loday. \textit{Realization of the Stasheff polytope.} Arch. Math. 83(3) (2004), 267-278.

\bibitem{Po05} Alexander Postnikov. \textit{Permutohedra, associahedra and beyond.}
Int.\ Math.\ Research Notices {2009} (2009), 1026-1106 

\bibitem{St63}James D. Stasheff. \textit{Homotopy associativity of $H$-spaces, I}, Transactions Amer. Math. Soc. {108} (1963), 275-292.

\bibitem{Zi} G\"unter M. Ziegler. \textit{Lectures on Polytopes.} Graduate Texts in Mathematics, Vol.~152. Springer-Verlag, New York, 1995.

\end{thebibliography}
\end{document}